\documentclass{amsart}

\numberwithin{equation}{section}

\usepackage{enumerate}

\usepackage[english]{babel}
\usepackage{amsfonts, amsmath, amsthm, amssymb,amscd,indentfirst,mathrsfs}

\usepackage{bm}
\usepackage{color}
\usepackage{xcolor}
\usepackage{tikz}
\usepackage{esint}

\usepackage[utf8]{inputenc} 
\usepackage[T1]{fontenc}

\usepackage{graphicx}
\usepackage{epstopdf}
\usepackage{latexsym}

\usepackage{palatino}
\usepackage{marginnote}

\newtheorem{theorem}{Theorem}[section]

\newtheorem{proposition}[theorem]{Proposition}
\newtheorem{lemma}[theorem]{Lemma}
\newtheorem{definition}[theorem]{Definition}

\newtheorem{remark}[theorem]{Remark}

	\def\bp{\begin{proof}}
		\def\ep{\end{proof}}

	\def\Sc{{\mathcal S}}

	\def\Lc{{\mathcal L}}

	\def\Ac{{\mathcal A}}

	\usepackage{amsaddr}

\begin{document}
		
		\title[Extrinsic black hole uniqueness]{Extrinsic black hole uniqueness in pure Lovelock gravity}
		
		\author{Levi Lopes de Lima and Frederico Gir\~ao}
		\address{Universidade Federal do Cear\'a (UFC),
			Departamento de Matem\'{a}tica, Campus do Pici, Av. Humberto Monte, s/n, Bloco 914, 60455-760,
			Fortaleza, CE, Brazil.}
		
		\author{Jos\'e Nat\'ario}
		\address{CAMGSD, Departamento de Matem\'atica, Instituto Superior T\'ecnico, 1049-001 Lisboa, Portugal}	
		
		\email{levi@mat.ufc.br,fred@mat.ufc.br,jnatar@math.tecnico.ulisboa.pt}
		
		\thanks{L.L. de Lima has been partially supported by CNPq/Brazil grant
			312485/2018-2, F. Gir\~ao  by  CNPq/Brazil grant 307239/2020-9  and J. Nat\'ario by  FCT/Portugal grant UIDP/MAT/04459/2020. The authors have also benefited from
			support coming from FUNCAP/CNPq/PRONEX grant 00068.01.00/15.
		}
		
		\begin{abstract}
			We define a notion of {\em extrinsic} black hole in pure Lovelock gravity of degree $k$ which captures the essential features of the so-called Lovelock-Schwarzs\-child solutions, viewed as rotationally invariant hypersurfaces with null $2k$-mean curvature in Euclidean space $\mathbb R^{n+1}$, $2\leq 2k\leq n-1$.  We then combine a regularity argument with a rigidity result by Ara\'ujo-Leite \cite{araujo2012two}  to prove, under a natural ellipticity condition, a global uniqueness theorem for this class of black holes. As a consequence we obtain,
			in the context of the corresponding Penrose inequality for graphs established by Ge-Wang-Wu \cite{ge2014new}, a local rigidity
			result for the Lovelock-Schwarzschild solutions. 
		\end{abstract}

		\maketitle


		\section{Introduction}\label{intro}

		Since the seminal work by Israel \cite{israel1967event}, the study of uniqueness properties of black hole solutions of Einstein field equations remains an active field of research; see \cite{heusler1996black,robinson2009four,chrusciel2012stationary} for recent reviews on this subject. From a physical viewpoint, this body of knowledge captures the accepted perception that ``black holes have no hair'', which roughly means that stationary black hole solutions satisfying suitable boundary conditions (both at the horizon and at spatial infinity) should be completely characterized by a few externally observable conserved charges (mass, angular momentum, etc.) For instance, in the static, asymptotically flat case, a celebrated result by Bunting and Masood-ul-Alam \cite{bunting1987nonexistence} guarantees that the black hole is completely determined by its ADM mass (and hence coincides with a Schwarzschild solution). A remarkable feature of this latter contribution is that its proof uses in an essential way the rigidity statement of the positive mass theorem due to Schoen and Yau \cite{schoen1981energy}. 
		
		In recent years, there has been a renewed interest in studying gravity theories in higher dimension $N=n+1$, $n\geq 4$, possibly coupled to certain kinds of matter fields (scalar, gauge, etc.), as these abstractions potentially qualify as low energy limits of certain versions of superstring theory. It turns out, however, that some of these theories, even if restricted to the Einsteinian setting, display exotic black hole solutions which can not be fully characterized by their conserved charges. In particular, black hole uniqueness, at least as classically envisaged, fails to hold; see \cite{emparan2008black,horowitz2012black,hollands2012black} for discussions of these phenomena. 
		
		The purpose of this note is to indicate that, at least for pure Lovelock gravity, a natural and elegant extension of Einstein's general relativity in higher dimensions, an appropriate version of black hole uniqueness may be restored under suitable assumptions.  As explained in Section \ref{state}, we consider asymptotically flat, time-symmetric vacuum solutions in pure Lovelock gravity whose initial data set carries   
		a compact (but not necessarily connected) inner horizon. These objects appear as extremal configurations in a conjectured Penrose inequality which so far has been checked in rather special cases \cite{ge2014gauss,ge2014new,li2014gauss,de2019gauss}. Notice that these competing solutions are not necessarily stationary, so our result somehow departs from the classical uniqueness theorems discussed above. On the other hand, our method of proof demands that the ``degree $k$'' Lovelock black holes we consider display the essential features of the so-called Lovelock-Schwarzschild solutions, viewed as embedded, rotationally invariant hypersurfaces with null $2k$-mean curvature in Euclidean space $\mathbb R^{n+1}$, $2\leq 2k\leq n-1$ (this explains the qualification ``extrinsic'' in the title). If we further assume that a certain ellipticity condition is satisfied everywhere, then our main result (Theorem \ref{main} below) confirms that the black hole solution is actually {\em congruent} to  a Lovelock-Schwarzschild solution. In particular, it is completely determined by its Gauss-Bonnet-Chern mass, a feature which clearly makes contact with the standard formulation of black hole uniqueness. As a consequence we obtain,
		in the context of the corresponding Penrose inequality for graphs established by Ge-Wang-Wu \cite{ge2014new}, a local rigidity
		result for the Lovelock-Schwarzschild solutions. 
		
		This note is organized as follows. In Section \ref{state} we review the pertinent aspects of pure Lovelock gravity of degree $k$, including the explicit description of the corresponding Lovelock-Schwarzschild black holes as  embedded, rotationally invariant hypersurfaces with null $2k$-mean curvature in Euclidean space $\mathbb R^{n+1}$, $2\leq 2k\leq n-1$. This section also contains the precise statement of our main result (Theorem \ref{main}), whose proof is presented in Section \ref{ref:proof}, with a preliminary discussion of the approach occurring in Section \ref{comm:all}. The main technical ingredient in the proof of Theorem \ref{main} is a remarkable rigidity result by Ara\'ujo and Leite \cite[Theorem 3.1]{araujo2012two}, which extends previous contributions in \cite{de2011two,hounie1999two,schoen1983uniqueness}. This result completely classifies {\em two-ended} extrinsic Lovelock black hole solutions and our reasoning essentially boils down to doubling the given one-ended Lovelock black hole solution across the horizon, which is assumed to lie in a fixed hyperplane $\Pi\hookrightarrow\mathbb R^{n+1}$. It turns out, however, that the so constructed two-ended black hole is in principle only of class $C^{1,1}$ along the horizon. It is precisely at this point that the ellipticity assumption comes into play, as it allows us to explicitly carry out a regularity argument in order to restore $C^2$-smoothness. Finally, Theorem \ref{rig:char} describes the Penrose-type local rigidity result mentioned above.

		\section{Pure Lovelock gravity and statement of the main result}\label{state}
		
		Starting with Einstein's relativity, it has been highly desirable to characterize  feasible gravity theories which have a Lorentzian metric $\overline g$ as its gravitational dynamical field by a few natural properties. According to this axiomatic approach, general covariance should lead to field equations of the type
		\begin{equation}\label{field:eq}
			\mathbf G(\overline g)_{ij}=\mathbf T_{ij},
		\end{equation}
		where the left-hand side is a twice covariant symmetric tensor whose (local) dependence on $\overline g$ must involve, in the simplest case we will be interested in, derivatives up to second order, whereas the right-hand side $\mathbf T$ collects the non-gravitat\-ional fields in the theory. Moreover, the usual (local) conservation law imposes a diverg\-ence-free condition on $\mathbf G$:
		\begin{equation}\label{div:cond}
			\nabla_{\overline g}^i\mathbf G(\overline g)_{ij}=0,
		\end{equation}
		where $\nabla_{\overline g}$ is the Levi-Civita connection associated to $\overline g$. A celebrated result by Lovelock \cite{lovelock1972four} confirms that in the physical dimension $3+1$, these conditions uniquely determine $\mathbf G(\overline g)$ as the well-known Einstein tensor plus a constant multiple of the metric:
		\begin{equation}\label{einstein:tensor}
			\mathbf G(\overline g)={\rm Ric}_{\overline g}-\frac{S_{\overline g}}{2}\overline g+\Lambda \overline g,
		\end{equation}
		where ${\rm Ric}_{\overline g}$ is the Ricci tensor of $\overline g$ and $S_{\overline g}={\rm tr}_{\overline g}{\rm Ric}_{\overline g}$ is the scalar curvature. Here and in the following we set the cosmological constant $\Lambda$ to vanish.
		
		However, in higher dimension $N=n+1$, $n\geq 4$, new phenomena emerge and this characterization of the Einstein tensor only remains true if we further assume that $\mathbf G$ depends {\em linearly} on the second derivatives $\partial^2\overline g$ of $\overline g$. More generally, symmetric tensors satisfying (\ref{div:cond})
		and with an arbitrary dependence on $\partial^2\overline g$ have been classified in \cite{lovelock1971einstein}; see \cite{navarro2011lovelock} for a modern account of this foundational result. It then follows that if we further assume that the dependence on $\partial^2\overline g$ is homogeneous of degree $k$ then $\mathbf G(\overline g)\approx \mathbf L_{2k}(\overline g)$, where the  {\em pure Lovelock tensor} ${\mathbf L_{2k}(\overline g)}$ is given by 
		\begin{equation}\label{lovel:local}
			{\mathbf L_{2k}(\overline g)}_{ij}=\overline g_{il}\delta^{li_1i_2\ldots i_{2k-1}i_{2k}}_{jj_1j_2\ldots j_{2k-1}j_{2k}}\overline R_{i_1i_2}^{j_1j_2}\ldots \overline R_{i_{2k-1}i_{2k}}^{j_{2k-1}j_{2k}}, \quad 1\leq k\leq \frac{n}{2}. 
		\end{equation}
		Here, $\overline R^{ij}_{kl}$ are the coefficients of the Riemann curvature tensor of $\overline g$ with respect to a local orthonormal basis of tangent vectors and the symbol $\approx$ relates quantities that possibly differ by a universal multiplicative constant. 
		
		\subsection{Pure Lovelock gravity}\label{pure:love:g}
		The field equations for pure Lovelock gravity {\em in vacuum}, namely,
		\begin{equation}\label{field:love:v}
			{\mathbf L_{2k}(\overline g)}=0,
		\end{equation}	
		arise from a variational principle involving a geometric Lagrangian density, the so-called  {\em Gauss-Bonnet curvature}
		\begin{equation}\label{gauss:bonnet}
			{\bf S}_{2k}(\overline g)\approx {\rm tr}_{\overline g} {\mathbf L_{2k}(\overline g)} \approx \delta^{i_1i_2\ldots i_{2k-1}i_{2k}}_{j_1j_2\ldots j_{2k-1}j_{2k}}\overline R_{i_1i_2}^{j_1j_2}\ldots \overline R_{i_{2k-1}i_{2k}}^{j_{2k-1}j_{2k}}.
		\end{equation}
		Thus, a metric $\overline g$ extremizes the  
		{Lovelock action}
		$$
		\overline g\longmapsto \int {\bf S}_{2k}(\overline g)
		$$
		if and only if (\ref{field:love:v}) is satisfied \cite{lovelock1989tensors,labbi2008variational}. 
		Notice that ${\bf S}_2(\overline g)$ is the full contraction of $\overline R$,  hence a multiple of $S_{\overline g}$. 
		Also, similarly to the Einstein tensor ($k=1$), the Lovelock tensor decomposes as 
		\[
		{{\bf L}}_{2k}(\overline g)\approx {{\bf Ric}}_{2k}(\overline g)-\frac{{\bf H}_{2k}(\overline g)}{2}\overline g,
		\]
		where ${{\bf Ric}}_{2k}(\overline g)$ is the so-called $2k$-{\em Ricci tensor}. In the language of double forms, one has ${{\bf Ric}}_{2k}(\overline g)= c_{\overline g}^{2k-1}\overline R^k$ and ${{\bf H}}_{2k}(\overline g)= c_{\overline g}^{2k}\overline R^k\approx {\bf S}_{2k}(\overline g)$, where $\overline R$ is the Riemann tensor of $\overline g$ and $c_{\overline g}$ is the contraction induced by $\overline g$ \cite[Definition 2.2]{labbi2008variational}.
		In particular, ${{\bf Ric}}_{2}(\overline g)={\rm Ric}_{\overline g}$.
		In this way we obtain a full-fledged gravitational theory (pure Lovelock gravity of degree $k$) which happens to be a natural generalization of Einstein's relativity {\em in vacuum} (this corresponds to $k=1$).

		\begin{remark}\label{dadhich}
		{\rm 	Another compelling indication that pure Lovelock is the right setup for gravity in higher dimensions comes from the writings of N. Dadhich and collaborators, where it is argued that this theory shares many of the congenial features of Einstein's relativity, including kinematicity \cite{dadhich2016distinguishing}, the existence of bound orbits around static subjects \cite{dadhich2013bound} and thermodynamical universality \cite{dadhich2012thermodynamical}; see \cite{dadhich2016distinguishing} and the references therein for more on these issues.}   
		\end{remark}

		\begin{remark}\label{intr:formul}
			{\rm It is useful to provide a coordinate free formulation of Lovelock theory. If $(\overline M^N,\overline g)$ is the underlying Lorentzian manifold, $N=n+1$, 
				recall that 
				the metric $\overline g$ induces a natural bundle isomorphism $T\overline M\equiv T^*\overline M$. Given a vector field $Z\in \Gamma(T\overline M)$ and a (local) volume element $\Omega$ one has, by Cartan's magic formula,
				$$
				{\rm d} {\bf i}_Z\Omega={\rm d} {\bf i}_Z\Omega+ {\bf i}_Z{\rm d}\Omega=\Lc_Z\Omega=({\rm div}_{\overline g}\,Z)\Omega,
				$$
				where ${\bf i}_Z$ is contraction with $Z$, $\Lc_Z$ is Lie derivative and ${\rm div}_{\overline g}$ is the divergence. Thus, the correspondence $Z\leftrightarrow \omega=i_Z\Omega$ gives rise to an isomorphism between $\Gamma(T^*\overline M)\equiv \Gamma(T\overline M)$ and $\Gamma(\Lambda^{n-1}T^*\overline M)$ such that ${\rm div}_{\overline g} Z=0$ if and only if ${\rm d}\omega=0$. Similarly, the correspondence
				$
				Z_1\otimes Z_2 \leftrightarrow {\bf i}_{Z_1}\Omega\otimes {\bf i}_{Z_2}\Omega
				$
				defines an isomorphism between $\Gamma(\otimes ^2T^*\overline M)$, the space of twice covariant tensors, and $\Gamma(\Lambda^{n}T^*\overline M\otimes \Lambda^{n}T^*\overline M)$, the space of $n$-forms with values in $n$-forms, which is well defined globally even if $\overline M$ is not orientable. Also, restriction to $\Sc^2(T^*\overline M)\subset \Gamma(\otimes ^2T^*\overline M)$, the space of symmetric $(2,0)$-tensors, defines an isomorphism between $\Sc^2(T^*\overline M)$ and $\Sc^2(\Lambda^{n}T^*\overline M)$, where, by definition, $\Sc^2(\Lambda^{p}T^*\overline M)$ is formed by those  $\eta\in\Gamma(\Lambda^{p}T^*\overline M\otimes \Lambda^{p}T^*\overline M)$ which are  symmetric in the sense that 
				\begin{equation}\label{def:symm}
					\eta(v_1\wedge\ldots\wedge v_{p}\otimes w_1\wedge\ldots\wedge
					w_{p})=\eta(w_1\wedge\ldots\wedge w_{p}\otimes v_1\wedge\ldots\wedge
					v_{p}).
				\end{equation}
				Also, a computation shows that $T\in \Sc^2(T^*\overline M)$ satisfies ${\rm div}_{\overline g}T=0$ if and only the corresponding $\eta\in \Sc^2(\Lambda^{n}T^*\overline M)$ satisfies ${\rm d}^{\nabla_{\overline g}}\eta=0$, where, as an operator acting on (\ref{def:symm}), ${\rm d}^{\nabla_{\overline g}}$ is the covariant exterior derivative defined by using the standard exterior derivative in the first factor and the covariant derivative $\nabla_{\overline g}$ induced by the Levi-Civita connection in the second factor. Now, $\overline g\in \Sc^2(\Lambda^{1}T^*\overline M)$ and $\overline R\in \Sc^2(\Lambda^{2}T^*\overline M)$, so we may  define
				\begin{equation}\label{lov}
					\widetilde{\mathbf  L}_{2k}(\overline g)= \underbrace{\overline R{\wedge}{\ldots} \wedge \overline R}_k\wedge \underbrace{\overline g\wedge \ldots \wedge \overline g}_{n-2k}\in \Sc^2(\Lambda^{n}T^*\overline M),\quad 1\leq k\leq \frac{n}{2}.
				\end{equation}
				Since ${\rm d}^{\nabla_{\overline g}}\overline g =0$ (metric compatibility) and ${\rm d}^{\nabla_{\overline g}} \overline R=0$ (Bianchi identity), we get ${\rm d}^{\nabla_{\overline g}} \widetilde{\mathbf  L}_{2k}(\overline g)=0$. Thus, the corresponding $(2,0)$-tensor $\widehat{\mathbf L}_{2k}(\overline g)\in S^2(T^*\overline M)$ satisfies ${\rm div}_{\overline g}\widehat{\mathbf L}_{2k}(\overline g)=0$. By Lovelock's theorem mentioned earlier, $\widehat{\mathbf L}_{2k}(\overline g)$ is a (universal) multiple of the  Lovelock tensor $\mathbf L_{2k}(\overline g)$.}
		\end{remark}

		\subsection{Lovelock-Schwarzschild black holes}\label{love:Sch}
		It turns out that the analogy of (pure) Lovelock gravity with Einstein relativity goes one step further, the reason being that the field equations (\ref{field:love:v})
		admit a one-parameter family of static, spherically symmetric black hole solutions \cite{crisostomo2000black,cai2006black,kastikainen2019quasi}. More precisely, in standard coordinates $(\mathsf t,r,\theta)$, $\theta\in\mathbb S^{n-1}$, the solution is given by
		\begin{equation}\label{black:hole}
			\overline g_{k,m}=-V_{k,m}(r)d\mathsf t^2+g_{k,m}, \quad g_{k,m}=\frac{dr^2}{V_{k,m}(r)}+r^2d\theta^2, \quad V_{k,m}(r)=1-\frac{2m}{r^{\frac{n}{k}-2}}, 
		\end{equation}
		where $m>0$ is a real parameter whose physical interpretation we discuss later on; see the discussion surrounding (\ref{egm:mass:for}). 
		Since the space-like slice $\mathsf t=0$ is totally geodesic, it follows that  the Riemannian metric $g_{k,m}$ satisfies the curvature condition
		\begin{equation}\label{slike:curv}
			{\bf S}_{2k}(g_{k,m})=0.
		\end{equation}
		Conversely, at least in the time-symmetric case, this curvature condition may be viewed as the constraint equation   
		a space-like initial data set should satisfy in order to have its Cauchy development yielding a solution of (\ref{field:love:v}) \cite{choquet1988cauchy,teitelboim1987dimensionally}.

		The black hole character of the metric $\overline g_{k,m}$ in (\ref{black:hole})  manifests itself in the fact that it displays an event horizon located at $r_{k,m}=(2m)^{{k}/{(n-2k)}}$. As in the classical situation, this suggests that the solution may be continued somehow beyond the horizon. Clearly, it suffices to perform this continuation for the space-like metric $g_{k,m}$ and this may be accomplished in  two rather distinct, but related, ways. First, we may work intrinsically and verify that 
		\begin{equation}\label{conf:rep}
			g_{k,m}=\left(1+\frac{m}{2\rho^{\frac{n}{k}-2}}\right)^{\frac{4k}{n-2k}}\left(d\rho^2+\rho^2d\theta^2\right),\quad \rho>0,
		\end{equation}
		which provides a conformal representation of the black hole metric with the horizon now located at $\rho_{k,m}=(m/2)^{k/(n-2k)}$. The corresponding lapse is 
		\[
		l(\rho)=\left(\frac{1-\frac{m}{2\rho^{\frac{n}{k}-2}}}{1+\frac{m}{2\rho^{\frac{n}{k}-2}}}\right)^2,
		\]
		which confirms the null character of the event horizon $\rho=\rho_{k,m}$. 
		We refer to the pair $(M_{k,m},g_{k,m})$, where $M_{k,m}:=(r_{k,m},+\infty)\times \mathbb S^{n-1}$, as a one-ended {\em Lovelock-Schwarzschild black hole} (associated to the pair $(k,m)$). When $k=1$ we recover the Schwarzschild solution of Einstein gravity.

		\subsection{Extrinsic Lovelock-Schwarzschild black holes}\label{ext:bh}
		For our purposes, it is convenient to work {\em extrinsically} and realize the  Lovelock-Schwarzschild black holes mentioned earlier by means of an isometric embedding $(M_{k,m},g_{k,m})\hookrightarrow (\mathbb R^{n+1},\overline \delta)$, where $\overline \delta$ is the flat metric. The key observation is that the passage from (\ref{lovel:local}) to (\ref{gauss:bonnet}), when adapted to a general Riemannian $n$-manifold $(M,g)$, implies that ${\bf S}_{2k}(g)={\rm tr}_g\mathbf L_{2k}(g)$. If we further assume that $(M,g)\hookrightarrow (\mathbb R^{n+1},\overline \delta)$ isometrically then by Gauss equation the curvature tensor $R$ of $g$ satisfies $R=\frac{1}{2}A\wedge A$, where  $A\in \Sc^2(\Lambda^1T^*M)$ is the shape operator of the embedding. In this Riemannian setting, (\ref{lov})  becomes
		$$
		\widetilde{\mathbf  L}_{2k}(g)= \frac{1}{2^k}\underbrace{A{\wedge} {\ldots} \wedge A}_{2k}\wedge \underbrace{g\wedge \ldots \wedge g}_{n-1-2k}, \quad 1\leq k\leq\frac{n-1}{2},
		$$  
		and a computation shows that ${\bf S}_{2k}(g)\approx{\bm \sigma}_{2k}(A)$, so that the time-symmetric constraint  (\ref{slike:curv}) becomes
		\begin{equation}\label{slike:curv2}
			{\bm \sigma}_{2k}(A)=0.
		\end{equation}
		Here, ${\bm \sigma}_p(A)={\bm \sigma}_p(\kappa_1,\ldots,\kappa_n)$, the $p$-{\em mean curvature} of $M$, is the elementary symmetric function of degree $p$ in the eigenvalues $\{\kappa_i\}$ of $A$ (the principal curvatures). This discussion suggests that each $g_{k,m}$ should be realized as the induced metric of some {\em rotationally invariant} isometric embedding of $M_{k,m}$ satisfying (\ref{slike:curv2}). We  now proceed to indicate how these invariant embeddings may be classified; compare with \cite{leite1990rotational}.   
		
		Consider the orthogonal group $O_n$ acting  isometrically on  $(\mathbb R^{n+1},\overline\delta)$ and leaving invariant the axis $x_{n+1}=t$. If $x_1=s$ is another coordinate axis orthogonal to $x_{n+1}$ at the origin then the half-plane $P=\{(s,t); s\geq 0\}$ is the orbit space of the $O_n$-action. If $\Sigma^n\hookrightarrow\mathbb R^{n+1}$ is {\em rotational}, i.e. $O_n$-invariant, then we denote by $\alpha\subset P$ its profile curve. We assume that $\alpha=\alpha(\tau)$ is parametrized by arc length: 
		\begin{equation}\label{unit}
			\dot s^2+ \dot t^2=1,
		\end{equation}
		where the dot stands for derivation with respect to $\tau$.
		If $\theta=(\theta_1,\cdots,\theta_{n-1})$ are standard coordinates in the spherical orbit then the induced metric on $\Sigma$ is
		\[
		g=d\tau^2+s(\tau)^2d\theta^2.
		\]
		From this we easily compute the sectional curvatures of $g$:
		\[
		K(\partial_{\theta_i},\partial_{\theta_j})=\frac{1-\dot s^2}{s^2},\quad K(\partial_{\theta_i},\partial_\tau)=-\frac{\ddot s}{s},
		\]
		which by Gauss equation give the principal curvatures of $\Sigma$:
		\[
		\kappa_1=\cdots=\kappa_{n-1}=\frac{\sqrt{1-\dot s^2}}{s}, \quad \kappa_n=-\frac{\ddot s}{\sqrt{1-\dot s^2}}.
		\]
		Thus, 
		\[
		{\bm \sigma}_p(A)=
		\left(
		\begin{array}{c}
			n-1\\
			p
		\end{array}
		\right)
		\left(
		\frac{\sqrt{1-\dot s^2}}{s}
		\right)^p-\left(
		\begin{array}{c}
			n-1\\
			p-1
		\end{array}
		\right)\frac{\ddot s}{\sqrt{1-\dot s^2}}\left(
		\frac{\sqrt{1-\dot s^2}}{s}
		\right)^{p-1},
		\]
		so that ${\bm \sigma}_{2k}(A)=0$, $2\leq 2k\leq n-1$, if and only if
		\begin{equation}\label{sigmap}
			(n-2k)(1-\dot s^2)-2k s\ddot s=0.
		\end{equation}
		
		We proceed by observing that the quantity 
		\[
		C(s,\dot s)=s^{\frac{n}{k}-2}(1-\dot s^2)
		\]
		satisfies $\dot C=0$, so $C$ is a first integral for (\ref{sigmap}).
		Thus, we are left with the task of classifying profile curves $\alpha=(s,t)$ satisfying
		\begin{equation}\label{sigma2k}
			s^{\frac{n}{k}-2}(1-\dot s^2)=2c, \quad c>0.
		\end{equation}
		By eliminating $\tau$ from (\ref{sigma2k})
		and (\ref{unit})
		we see that $t=t(s)$ satisfies
		\begin{equation}\label{fin}
			\left(\frac{dt}{ds}\right)^2=
			\frac{2c}{{s^{\frac{n}{k}-2}}-2c}, \quad s> (2c)^{\frac{k}{n-2k}}.
		\end{equation}
		This shows that $\Sigma$ is the union of two vertical graphs $\Sigma^\pm$ associated to functions $t=t^{\pm}$ satisfying (\ref{fin}). Since $|(dt^{\pm}/ds)((2c)^{k/(n-2k)})|=+\infty$, these graphs meet together along the hyperplane $\mathbb R^n=\{x_{n+1}=0\}$ and they both intersect the hyperplane orthogonally.
		Notice that $\Sigma$ is smooth along the common horizon because (\ref{fin}) easily implies that \begin{equation}\label{com:sm}
			\frac{d^rs}{d(t^\pm)^r}\approx s^{\frac{n}{k}-1-r},\quad  r\geq 2,
		\end{equation}
		which remains finite and converges to a common value as $t^\pm\to 0$. 
		In these coordinates, the induced metric on $\Sigma=\Sigma^+\cup\Sigma^-$ is 
		\begin{equation}\label{asymp:beh}
			g_{k,c}=\frac{ds^2}{ V_{k,c}(s)}+s^2d\theta^2, \quad  V_{k,c}(s)={{1-\frac{2c}{s^{\frac{n}{k}-2}}}}.
		\end{equation}
		Thus, if we set $s=r$ and $c=m$ we see that $(\Sigma^+,g_{k,m})$ recovers the one-ended Lovelock-Schwarzschild black hole $(M_{k,m},g_{k,m})$ as the induced geometry on the corresponding $O_n$-invariant graph 
		satisfying (\ref{slike:curv2}). Also, $(\Sigma,g_{k,m})$ reproduces the two-ended conformal representation in (\ref{conf:rep}) with the horizon being located at the intersection $\Sigma\cap\mathbb R^n$.
		Notice that from (\ref{asymp:beh}) we easily check that, as $s\to+\infty$,  
		\begin{equation}\label{asymp:beh2}
			(g_{k,m})_{ij}=\left(1 +2m s^{-\frac{n}{k}+2}\right)\delta_{ij}+O(s^{-\frac{n}{k}+1}),	
		\end{equation}
		where $\delta$ is the flat metric in $\mathbb R^n$.
		
		\subsection{Extrinsic Lovelock black holes and the main theorem}\label{ext:main}
		After this preliminary discussion, we finally describe the class of {\em extrinsic} black holes we are interested in, which appear prominently in our main result (Theorem \ref{main} below). The idea is to consider a one-ended abstract black hole $(M,g)$ which can be isometrically embedded in $(\mathbb R^{n+1},\overline\delta)$ in such a way that it shares the essential features of the extrinsic model $(\Sigma^+,g_{k,m})$
		described above. More precisely, we require the following from $(M,g)$: 
		\begin{enumerate}[I.]
			\item from the abstract viewpoint, the $n$-manifold $M$ carries a (a not necessarily connected) compact inner boundary $\Gamma$ and a unique end, say $\mathcal E$, which is diffeomorphic to the complement of a ball in $\mathbb R^n$; 
			\item under the isometric embedding $(M,g)\hookrightarrow (\mathbb R^{n+1},\overline\delta)$, $M$ lies entirely inside the half-space $x_{n+1}\geq 0$, $\Gamma$ lies in $\mathbb R^n$, and  $M\cap \mathbb R^n=\Gamma$, with the intersection being orthogonal along $\Gamma$; 
			\item besides ${\bm \sigma}_{2k}(A)=0$, we also assume that ${\bm \sigma}_{2k+1}(A)\neq 0$ everywhere along $M$, where $A$ is the shape operator of the embedding; 
			\item the end $\mathcal E$ may be written as the graph of a smooth function $u:\mathbb R^n\backslash \{|x|\leq R\}\to \mathbb R$, $R$ large. Moreover, as $|x|\to +\infty$, $u$ displays the {\em same} asymptotic behavior as $t^+$ defined by (\ref{fin}) under the identification $s=|x|$.
		\end{enumerate}  
		
		Notice that (II) implies that $\Gamma\hookrightarrow M$ is totally geodesic and hence qualifies to be a horizon. In (III), besides the constraint equation ${\bm \sigma}_{2k}(A)=0$ (compare with (\ref{slike:curv}) and (\ref{slike:curv2})) we also require the ``ellipticity condition'' ${\bm \sigma}_{2k+1}(A)\neq 0$; see Section  \ref{comm:all} below for more on this point. We emphasize that these conditions are met by our model $(M_{k,m},g_{k,m})$ viewed as an $O_n$-invariant hypersurface in $\mathbb R^{n+1}$. 
		
		We now elaborate on the prescribed asymptotic behavior of $u$ as  indicated in (IV).  For this it is convenient to set 
		\[
		q=q_{k,n}=\frac{n}{2k}-1 \Longrightarrow \frac{1}{2k}\leq q\leq \frac{n}{2}-1.
		\] 
		We then consider the following asymptotic regimes mentioned in (IV), which can be read off from the corresponding asymptotic expansion of $t$ in (\ref{fin}); compare with Definitions 3.1 and 4.2 in \cite{araujo2012two}.
		\begin{itemize}
			\item if $q>1$ then
			\begin{equation}\label{cond1}
				u(x)=a_1-\frac{a}{q-1}|x|^{1-q}+\langle c,x\rangle|x|^{-q-1}+O(|x|^{-q-1});
			\end{equation}
			\item if $q=1$ then  
			\begin{equation}\label{cond2}
				u(x)=a_1+a\log|x|+\langle c,x\rangle|x|^{-2} +O(|x|^{-2});
			\end{equation}
			\item if $(2k)^{-1}\leq q<1$ we decompose
			\begin{eqnarray*}
				\left[\frac{1}{2k},1\right) 
				& = & \left\{\frac{1}{2k}\right\}\cup \left(\frac{1}{2k},\frac{1}{2k-1}\right)\cup\left\{\frac{1}{2k-1}\right\}\cup
				\left(\frac{1}{2k-1},\frac{1}{2k-2}\right]\cup \cdots\\
				& & \quad \cdots \cup \left\{\frac{1}{3}\right\}\cup \left(\frac{1}{3},\frac{1}{2}\right]\cup \left(\frac{1}{2},1\right).
			\end{eqnarray*}
			To match the notation in \cite[Section 4]{araujo2012two}, we label the intervals in this decomposition as 
			\[
			I_m^1=\left(\frac{1}{2m+2},\frac{1}{2m+1}\right), \quad I_m^2=\left(\frac{1}{2m+3},\frac{1}{2m+2}\right], \quad 0\leq m\leq k-1,
			\]
			whereas the points  are labeled by $q_{m}=(2m+3)^{-1}$, $0\leq m\leq k-3/2$. With this notation, the asymptotics of $u$ in this range of $q$ may be prescribed as follows.
			\begin{enumerate}
				\item if $q\in I^1_m$ then
				\begin{equation}\label{cond3}
					u(x)=a_1+P_m(a,|x|)+\langle c ,x\rangle|x|^{-1-q}+O(|x|^{1-(2m+3)q});
				\end{equation}
				\item if $q\in I^2_m$ then
				\begin{eqnarray}\label{cond4}
					u(x) & = & a_1+P_m(a,|x|)+a_2|x|^{1-(2m+3)q}\nonumber \\ 
					& & \quad + \langle c ,x\rangle|x|^{-1-q}+O(|x|^{1-(2m+5)q});
				\end{eqnarray}
				\item if $q=q_m$ then
				\begin{eqnarray}\label{cond5}
					u(x) & = & a_1+P_m(a,|x|)+B_{m+1}a^{1/q}\log |x|\nonumber \\
					& & \quad + \langle c ,x\rangle|x|^{-1-q}+O(|x|^{-2q}).
				\end{eqnarray}
			\end{enumerate}
			Here, for $(2j+1)q<1$ we set 
			\[
			P_j(a,|x|)=C_0a|x|^{1-q}+C_1a^3|x|^{1-3q}+\cdots + C_ja^{2j+1}|x|^{1-(2j+1)q},
			\]
			where
			\[
			C_j=\frac{B_j}{1-(2j+1)q}, \quad 
			B_j=\frac{1\times 3\times\cdots\times (2j-1)}{2^j\times j!}. 
			\]
			Let us agree that, in these expansions, $a>0$, $a_1,a_2\in\mathbb R$ and $c\in\mathbb R^n$. Notice also that $C_0=(1-q)^{-1}$.
		\end{itemize}
		
		An important point in these expansions is that an examination of the (non-constant) leading term, which equals $a(1-q)^{-1}|x|^{1-q}$ for $q\neq 1$ and $a\log |x|$ for $q=1$,  shows that the induced metric is given by
		\begin{eqnarray}\label{cgb:asymp}
			g_{ij} & = & \delta_{ij}+\frac{\partial u}{\partial x_i}\frac{\partial u}{\partial x_j}\nonumber\\
			& = &  \delta_{ij}+a^2|x|^{-2q-2}x_ix_j+O(|x|^{-2q-1})\nonumber \\
			& = & \delta_{ij}+a^2 O(|x|^{-2q}),
		\end{eqnarray}
		which matches with (\ref{asymp:beh2}).  
		In fact, a comparison of these expressions shows that 
		\begin{equation}\label{a:m}
			a^2=2m.
		\end{equation}

		\begin{definition}\label{ext:bh:def}
			We say that a pair $(M,g)$ is an {\em extrinsic Lovelock black hole} if there exists an isometric embedding $(M,g)\hookrightarrow(\mathbb R^{n+1}, \overline \delta)$ satisfying  conditions (I)-(IV) above (for some $k$ such that $2\leq 2k\leq n-1$).  
		\end{definition}
		
		Thus, by (\ref{a:m}) an extrinsic Lovelock black hole has the same asymptotic expansion as the model Lovelock-Schwarzschild black hole (at least up to first order). 
		In any case, 
		with this definition at hand we may finally state our extrinsic black hole uniqueness theorem in pure Lovelock gravity. 
		
		\begin{theorem}\label{main}
			Let $(M,g)\hookrightarrow(\mathbb R^{n+1},\overline \delta)$ be an extrinsic Lovelock black hole as in Definition \ref{ext:bh:def} above (for some $k$). Then $(M,g)$ is congruent to the Lovelock-Schwarzschild black hole $(M_{k,m},g_{k,m})$ viewed as an $O_n$-invariant hypersurface in $\mathbb R^{n+1}$ (for some $m>0$).
		\end{theorem}

		\section{Preliminary remarks on the proof of Theorem \ref{main}}\label{comm:all}
		
		The key technical ingredient in the proof of Theorem \ref{main} is a remarkable result by Ara\'ujo and Leite \cite[Theorem 3.1]{araujo2012two}, which extends previous contributions in \cite{de2011two,hounie1999two,schoen1983uniqueness}. More precisely, they classify complete, embedded two-ended hypersurfaces  $M'\hookrightarrow\mathbb R^{n+1}$ such that:
		\begin{enumerate}[a.]
			\item its $2k$-mean curvature vanishes (${\bm \sigma}_{2k}(A)=0$);
			\item it is elliptic in the sense that ${\bm \sigma}_{2k+1}(A)\neq 0$ everywhere;
			\item each of its ends is asymptotically rotationally symmetric (they can be be written as a graph over the complement of a ball in a hyperplane $\Pi\hookrightarrow\mathbb R^{n+1}$ associated to a function, say $u$, which behaves precisely as in (\ref{cond1})-(\ref{cond5}), depending on the value of $q=q_{k,n}$). Ends with this property are termed {\em regular} in \cite{araujo2012two}.
		\end{enumerate} 
		
		We now elaborate a bit on these assumptions. We first note that the null $2k$-mean curvature condition in (a) is the Euler-Lagrange equation associated to   
		a natural variational problem, namely, that related to the functional $\int\sigma_{2k-1}(A)$ \cite{reilly1973variational}.  In this variational setting, the corresponding Jacobi operator is given by
		$$
		J_k={\rm div}_g(N_{2k-1}(A)\nabla\cdot)-(2k+1){\bm \sigma}_{2k+1}(A),
		$$
		where
		\[
		N_p(A)={\bm \sigma}_p(A)I-{\bm \sigma}_{p-1}(A)A+\cdots +(-1)^pA^p
		\]
		is the so-called Newton tensor. 
		It turns out that the ellipticity condition ${\bm \sigma}_{2k+1}(A)\neq 0$ in (b) is equivalent to $N_{2k-1}(A)$ being positive or negative definite, which means of course that  $J_k$ is  elliptic as a differential operator. Ellipticity and embededness are key assumptions in \cite{araujo2012two} as their argument relies heavily on Aleksandrov Reflection Principle. As for (c), we first note that   
		the asymptotics (\ref{cgb:asymp}) guarantees that each regular end as above has a so-called {\em Gauss-Bonnet-Chern mass} $\mathfrak m_k(g)$ attached to it \cite{ge2014new,ge2014gauss}. A direct computation shows that 
		\begin{equation}\label{egm:mass:for}
			\mathfrak m_k(g)\approx a^{2k}\approx m^k.
		\end{equation}
		From our extrinsic viewpoint, this invariant may be accessed as follows. For each end $\mathcal E$ and for each smooth cycle $C\subset \mathcal E$ generating $H_{n-1}(\mathcal E)$ we may consider the flux quantity 
		\[
		{\rm Flux}(\mathcal E)=\int_{C}\langle N_{2k-1}(A)\xi,\nu\rangle d{\rm vol}_C, 
		\] 
		where $\nu$ is the outward unit normal vector field to $C\hookrightarrow \mathcal E$ and $\xi$ is the unit normal vector field to the hyperplane $\Pi$ over which $\mathcal E$ is graphically represented as in (c). We now recall that any coordinate function $x_i$, when restricted to $M'$, satisfies 
		\[
		{\rm div}_g(N_{2k-1}(A)\nabla x_i)=0;
		\]
		see \cite[Theorem C]{reilly1973variational}.
		Thus, an application of the divergence theorem shows 
		that ${\rm Flux}(\mathcal E)$ is a homological invariant of $C$ and hence may be computed in the asymptotic limit by means of the expansion for $u$ in (\ref{cond1})-(\ref{cond5}). According to \cite[Theorem 2.4]{araujo2012two}, the final result is 
		\begin{equation}\label{final:flux}
			{\rm Flux}(\mathcal E)\approx a^{2k},
		\end{equation}
		which shows by (\ref{egm:mass:for}) that $\mathcal E$ has the same flux (or Gauss-Bonnet-Chern mass) as the one-ended Lovelock-Schwarzs\-child black hole to which it is asymptotic. 
		This flux formula is the very first step in the proof of the main result in \cite{araujo2012two}. Indeed, when the given hypersurface has exactly two ends, then (\ref{final:flux}) and the balance of fluxes imply that these ends are parallel to each other with opposite orientations and the {\em same} mass $\mathfrak m_k$ as in (\ref{egm:mass:for}). The Reflection Method is then used twice: first to guarantee that the hypersurface is symmetric across some ``horizontal'' hyperplane (this involves adjusting the constant terms in the expansions) and then, after locating the ``vertical'' axis of symmetry (which involves choosing coordinates so that $c=0$),  to make sure that the hypersurface is symmetric across hyperplanes containing this axis. This argument leads to the next remarkable rigidity assertion, which follows from the main results in \cite{araujo2012two}.
		
		\begin{theorem}\label{al:rem:res}\cite[Theorems 3.5 and 4.5]{araujo2012two}
			Any complete, embedded hypersurface $M'\hookrightarrow\mathbb R^{n+1}$ with two regular ends and satisfying ${\bm \sigma}_{2k}(A)=0$ and ${\bm \sigma}_{2k+1}(A)\neq 0$ everywhere (for some $k$ such that $2\leq 2k\le n-1$) is congruent to some  two-ended extrinsic Lovelock-Schwarzs\-child black hole $(\Sigma,g_{k,m})$.
		\end{theorem}

		We now briefly sketch the proof of Theorem \ref{main}, which is detailed in Section \ref{ref:proof}. We first reflect a given one-ended extrinsic Lovelock black hole $(M,g)\hookrightarrow(\mathbb R^{n+1},\overline \delta)$ across the hyperplane  containing the horizon
		$\Gamma$ so as to obtain a two-ended hypersurface $M$ satisfying all the requirements in Theorem \ref{al:rem:res}, except that in principle it is only $C^{1,1}$
		along $\Gamma$; compare this to the discussion surrounding (\ref{com:sm}), which guarantees that the Lovelock-Schwarzschild solution is smooth under reflection across the horizon. We then use the ellipticity
		condition to prove a regularity result (Proposition \ref{regprop}) showing that actually
		$M'$ is $C^2$ (in fact, smooth) along the horizon. Thus, Theorem \ref{al:rem:res} now applies and we conclude that the original Lovelock black hole is congruent to some one-ended Lovelock-Schwarzschild black hole $(M_{k,m},g_{k,m})$, which completes the argument.

		\section{A regularity result and the proof of Theorem \ref{main}}
		\label{ref:proof}

		As explained above, the proof of Theorem \ref{main} involves the consideration of the {embedded} $C^{1,1}$ hypesurface $M'$ obtained from our one-ended Lovelock black hole after reflection across the hyperplane containing the horizon $\Gamma$. More precisely, Theorem \ref{main} follows immediately from the rigidity result in Theorem \ref{al:rem:res} if we are able to show that $M'$ is actually of class $C^2$ along $\Gamma$. Since the argument is local, we fix $q\in \Gamma$ and write locally $M'$ around $q$ as the graph of a $C^{1,1}$ function $v$ defined in a small neighborhood $U$ of the origin $0\in T_pM'$. Choose rectangular coordinates $(y_1,\cdots,y_n)$ in $U$ so that the hypersurface $\Gamma_0\subset U$ defined by $y_n=0$ is such that $v|_{\Gamma_0}$ is the graph representation of $\Gamma$. Notice that $\Gamma_0$ determines a decomposition $U=U^+\cup U^-$, where $U^+$ (respectively, $U^-$) is given by $y_n\geq 0$ (respectively, $y_n\leq 0$). Clearly, $U^+\cap U^-=\Gamma_0$. We also set $v^{\pm}=v|_{U^{\pm}}$. Moreover, let us agree on the index ranges $1\leq i,j,\cdots \leq n$ and $1\leq \alpha,\beta,\cdots\leq n-1$.
		
		We now observe that the following properties hold:
		\begin{itemize}
			\item The partial derivatives $v^{\pm}_i$ are $C^1$ along $\Gamma_0$ with $v_i^+=v_i^-$ there;
			\item The function $v^\pm$ is $C^2$ on $U^\pm$ and $v^+_{\alpha\beta}=v^-_{\alpha\beta}$ along $\Gamma_0$.
		\end{itemize}
		Let us agree that, when decorating $v$, subscripts corresponding to partical differentiation
		with respect to $y$.
		
		These properties entail the following facts. First, the second property implies that, as we approach $\Gamma_0$ by interior points of $U^{\pm}$, all second order derivatives $v_{ij}^\pm$ exist in the limit and are continuous on $U\pm$. The point here is to check whether these derivatives agree along $\Gamma_0$ for each $(i,j)$, so that $v$ is indeed $C^2$ on $U$, which implies that $M'$ is $C^2$  by the fact that $q$ has been arbitrarily chosen.
		We already know that $v^+_{\alpha\beta}=v^-_{\alpha\beta}$ and, moreover, by the content of the first property applied to $v_n$, we see that $v^+_{\alpha n}=v^-_{\alpha n}$
		along $\Gamma_0$  as well. Thus, we are led with the task of checking whether $v^+_{nn}=v^-_{nn}$ along $\Gamma_0$.
		
		We notice that $M'_{\pm}=v(U^\pm)$ both have a well-defined shape operator, say $A^{\pm}$, with the usual properties (symmetry, etc.) holding up to $\Gamma_0$. We note for further reference that, in nonparametric coordinates,
		\begin{equation}\label{nonparshape}
			A_{ij}^\pm=B_{ij}^\pm+C_{ij}^\pm,
		\end{equation}
		where
		\begin{equation}\label{nonparshape2}
			B_{ij}^\pm=\frac{v^\pm_{ij}}{W},\quad C_{ij}^\pm=-\frac{1}{W^3}\sum_kv^\pm_iv^\pm_kv^\pm_{kj},\quad W=\sqrt{1+|\nabla v^{\pm}|^2}.
		\end{equation}
		
		As usual, given a symmetric matrix $\Ac$, we denote by ${\bm \sigma}_p(\Ac)$ the elementary symmetric function of degree $p$ in the eigenvalues of $\Ac$. In particular, we set ${\bm \sigma}_p(u^\pm):={\bm \sigma}_p(A^{\pm})$, so that the following property follows from the assumptions of Theorem \ref{main} and the way $M'$ was constructed from $M$:

		\begin{itemize}
			\item $v^\pm$ is an elliptic solution of ${\bm \sigma}_{2k}(v^\pm)=0$ in the sense that ${\bm \sigma}_{2k+1}(v^\pm)\neq 0$.
		\end{itemize}

		The following proposition provides the regularity result we are looking for.

		\begin{proposition}\label{regprop}
			Under the conditions above, $v^+_{nn}=v^{-}_{nn}$ along $\Gamma_0$. In particular, $M'$ is of class $C^2$ (in fact, smooth).
		\end{proposition}
		
		We start the proof by observing that, in general, ${\bm \sigma}_p(\Ac)$ is the sum of the principal minors of order $p$ of the symmetric matrix $\Ac$, so that
		\begin{equation}\label{reil}
			{\bm \sigma}_{2k}(\Ac)=\frac{1}{(2k)!} \sum\delta^{i_1\ldots i_{2k}}_{j_1\ldots
				j_{2k}}\Ac_{i_1j_1}\ldots \Ac_{ i_{2k}j_{2k}}.
		\end{equation}
		It then follows from
		(\ref{nonparshape}) that
		\begin{equation}\label{minor2}
			{\bm \sigma}_{2k}(v^\pm)={\bm \sigma}_{2k}(B^\pm)+{\bm \sigma}_{2k,\geq 1}(B^\pm,C^\pm),
		\end{equation}
		where
		\begin{equation}\label{minor3}
			{\bm \sigma}_{2k,\geq 1}(B^\pm,C^\pm)=	\sum_{r=1}^{2k}{\bm \sigma}_{2k,r}(B^\pm,C^\pm)
		\end{equation}
		and
		\begin{equation}\label{sus}
			{\bm \sigma}_{2k,r}(B^\pm,C^\pm)=\frac{1}{(2k)!} \sum\delta^{i_1\ldots
				i_{2k}}_{j_1\ldots j_{2k}}C^\pm_{i_1j_1}\ldots
			C^\pm_{i_rj_r}B^\pm_{i_{r+1}j_{r+1}}\ldots B^\pm_{i_{2k}j_{2k}}.
		\end{equation}
		We now observe that, due to $v_i^\pm(0)=0$, the ellipticity condition implies that the matrix
		\begin{equation}\label{ij}
			\frac{\partial {\bm \sigma}_{2k}(v^\pm)}{\partial v^\pm_{ij}}(0)=
			\frac{\partial {\bm \sigma}_{2k}(B^\pm)}{\partial v^\pm_{ij}}(0)
		\end{equation}
		is positive or negative definite (see \cite[Section 1]{hounie1999two} for a clarification of this point).
		
		\begin{lemma}\label{non:zero}
			Define 
			\[
			D^\pm=\frac{1}{(2k)!W^{2k}} \sum\delta^{\alpha_1\ldots \alpha_{2k-1}}_{\beta_1\ldots
				\beta_{2k-1}}v_{\alpha_1\beta_1}\ldots v_{ \alpha_{2k-1}\beta_{2k-1}}, 
			\]
			where the index $n$ is not allowed on the right-hand side.
			Then $D^\pm(0)\neq 0$. 
		\end{lemma}
		
		\begin{proof}
			Note that ${\bm \sigma}_{2k,\geq 1}(B^\pm,C^\pm)$ always carries at least a first derivative of $v^\pm$. Thus, $v_i^\pm(0)=0$ implies that 
			\[
			\frac{\partial}{\partial v^\pm_{nn}}{\bm \sigma}_{2k,r\geq 1}(B^\pm,C^\pm)(0)=0. 
			\]
			From (\ref{minor2}) we then get
			\begin{equation}\label{diff}
				\frac{\partial}{\partial v^\pm_{nn}}{\bm \sigma}_{2k}(v^\pm)(0)=\frac{\partial}{\partial v^\pm_{nn}}{\bm \sigma}_{2k}(B^\pm)(0)=D^\pm(0), 
			\end{equation}
			where in the last step we used that 
			\begin{equation}\label{furt:dec}
				{\bm \sigma}_{2k}(B^\pm)=D^\pm v_{nn}+{\bm \sigma}^\dagger_{2k}(B^\pm),
			\end{equation}
			where ${\bm \sigma}^\dagger_{2k}(B^\pm)$ collects the terms {\em not} depending on $v_{nn}$. Now, observe 
			that the left-hand side in (\ref{diff}) is precisely the $(n,n)$-entry of (\ref{ij}).
		\end{proof}

		We now use that ${\bm \sigma}_{2k}(v^\pm)=0$ in (\ref{minor2}), which together with (\ref{furt:dec}) leads to 
		\begin{equation}\label{furt:2}
			D^\pm v^\pm_{nn}+{\bm \sigma}^\dagger_{2k}(B^\pm)+{\bm \sigma}_{2k,\geq 1}(B^\pm,C^\pm)=0.
		\end{equation}
		It follows from (\ref{diff}) that $D^\pm$ does not vanish and remains bounded in a neighborhood of the origin, so that
		$$
		v_{nn}^\pm=-\frac{{\bm \sigma}^\dagger_{2k}(B^\pm)}{D^\pm}-\frac{{\bm \sigma}_{2k,\geq 1}(B^\pm,C^\pm)}{D^\pm}
		$$
		in this neighborhood. Due to the fact that ${{\bm \sigma}_{2k,\geq 1}(B^\pm,C^\pm)}$ depends at least quadratically on the first order derivatives, the last term on the right-hand side vanishes as we approach the origin, since all second order derivatives, including $v^\pm_{nn}$, remain bounded there. On the other hand, the first term on the right-hand side only depends on $v^\pm_i$ and $v^{\pm}_{\alpha i}$ and, since they coincide at the origin, we conclude that $v_{nn}^+(0)=v_{nn}^-(0)$, as desired.
		This completes the proof of Proposition \ref{regprop} and hence, by the comments in the end of the previous section,  of Theorem \ref{main}.

		The ellipticity of the Lovelock-Schwarzschild black holes $(M_{k,m},g_{k,m})$ leads to a nice  application of Theorem \ref{main}, which we now  describe. We first observe that for $(M_{k,m},g_{k,m})$ the precise location $r_{k,m}=(2m)^{k/(n-2k)}$ of the horizon allows us to solve for the mass $\mathfrak m_k(g_{k,m})$ in terms of the area $A$. More precisely, there holds 
		\begin{equation}\label{mass:area}
			\mathfrak m_k(g_{n,k})= c_{k,n}A^{\frac{n-1}{n-2k}}, \quad A={\rm Area}_{n-1}(M_{k,m},g_{k,m}),
		\end{equation}
		where $c_{k,n}>0$  is a certain universal constant. This motivates the following version of the Penrose inequality in Lovelock gravity, which has been proved in \cite{ge2014new}; see also \cite{li2014gauss,de2019gauss} for the higher co-dimensional case. This extends to the Lovelock setting previous contributions in \cite{lam2010graphs,de2015adm}, which handled the classical case $k=1$.
		
		\begin{theorem}\label{pen:love}
			Let $(M,g)$ satisfy (I) and  
			(II) above and assume further that
			\begin{enumerate}
				\item $(M,g)$ is globally a graph over a ``horizontal'' hyperplane $\Pi$ containing the horizon $\Gamma$, which is assumed to be convex;
				\item the graphing function in the previous item behaves as in item (IV) of Subsection \ref{ext:main}, so that the induced metric satisfies the asymptotics (\ref{asymp:beh2});
				\item the energy  condition ${\bm \sigma}_{2k}(A)\geq 0$ is satisfied everywhere.
			\end{enumerate} 
			Then the following Penrose-type inequality holds:
			\begin{equation}\label{pen:love:2}
				\mathfrak m_k(g)\geq c_{k,n}{\rm Area}_{n-1}(\Gamma)^{\frac{n-1}{n-2k}}.
			\end{equation}
		\end{theorem}
		
		Given this lower bound for the Gauss-Bonnet-Chern mass in terms of the area, a fundamental question is to determine which configurations attain the equality. It follows from the arguments leading to (\ref{pen:love:2}) that this necessarily implies that ${\bm \sigma}_{2k}(A)=0$. In general, this extra piece of information does not suffice to say much about the corresponding $(M,g)$, except in case $k=1$, where this  question has been treated in \cite{huang2015equality,de2012rigidity}. However, at least in case $(M,g)$ is a small perturbation of $(M_{k,m},g_{k,m})$, Theorem \ref{main} applies to yield the expected characterization. 
		
		\begin{theorem}\label{rig:char}
			Let $(M,g)$ be as in Theorem \ref{pen:love} and assume further that ${\bm \sigma}_{2k}(A)=0$ everywhere. Also, assume that $M$ is a sufficiently small $C^2$ perturbation of
			an extrinsic Lovelock-Schwarzschild black hole $(M_{k,m},g_{k,m})$ with the same asymptotics. Then $(M,g)$ is congruent to $(M_{k,m},g_{k,m})$. 
		\end{theorem}
		
		\begin{proof}
			Given that $(M_{k,m},g_{k,m})$ is elliptic,
			the result follows from Theorem \ref{main} and the obvious fact that the assumptions $M\cap \Pi=\Gamma$ and ${\bm \sigma}_{2k+1}(A)\neq 0$ are both preserved under such small $C^2$ perturbations. 
		\end{proof}

		\section{Further comments on the  case $\Lambda<0$}\label{further}
		
		We may also consider a version of pure Lovelock gravity in the presence of a {\em negative} cosmological constant, so that (\ref{field:love:v}) gets replaced by
		\[
		{\bf L}_{2k}(\overline g)+\Lambda\overline g=0, \quad \Lambda<0.
		\] 
		After a proper normalization of $\Lambda$, this field equation admits a vacuum black hole solution, the so-called {\em Lovelock-adS metric} given by 
		\[
		\overline g_{k,m,-}=-V_{k,m,-}(\mathfrak r)d\mathfrak t^2+g_{k,m,-},\quad m>0,
		\]
		where the space-like slice is
		\[
		g_{k,m,-}=\frac{d\mathfrak r^2}{V_{k,m,-}(\mathfrak r)}+\mathfrak r^2d\theta^2, \quad V_{k,m,-}(\mathfrak r)=1 +\mathfrak r^2-\frac{2m}{\mathfrak r^{\frac{n}{k}-2}}; 
		\] 
		compare with (\ref{black:hole}). Notice that $g_{k,m,-}$ is defined on the manifold $M_{k,m,-}=(\mathfrak r_{k,m,-},+\infty)\times\mathbb S^{n-1}$, where $\mathfrak r_{k,m,-}$ is the  positive zero of $V_{k,m,-}(\mathfrak r)=0$. Hence, as expected, this solution displays a horizon located as $\mathfrak r=\mathfrak r_{k,m,-}$. Moreover,  as in the $\Lambda=0$ case, the constant $m^k$ may be identified to the Gauss-Bonnet-Chern mass of the asymptotically hyperbolic manifold $(M_{k,m,-}, g_{k,m,-})$ as defined in \cite{ge2015gbc}.

		We observe that the space-like slice $(M_{k,m,-}, g_{k,m,-})$ above may be embedded as an asymptotically hyperbolic graph in hyperbolic space $(\mathbb H^{n+1}, \overline g_{0,-})$, where 
		\[
		\overline g_{0,-}=\left(1 +\mathfrak r^2\right)d\mathfrak t^2+g_{0,-},\quad g_{0,-}=\frac{d\mathfrak r^2}{1 +\mathfrak r^2}+\mathfrak r^2d\theta^2.
		\] 
		Similarly to (\ref{fin}), we now have
		\[
		\left(\frac{d\mathfrak t}{d\mathfrak r}\right)^2=\frac{2m}{(1+\mathfrak r^2)^2(\mathfrak r^{\frac{n}{k}-2}+\mathfrak r^{\frac{n}{k}}-2m)},
		\]
		so that solving  for $\mathfrak t=\mathfrak t(\mathfrak r)$ defines the (two-sheeted) graph. 
		A direct computation, quite similar to the one appearing in Subsection \ref{ext:bh}, confirms that this graphical realization of the Lovelock-adS black holes exhaust, up to an isometry, the family of rotationally invariant hypersurfaces in $\mathbb H^{n+1}$ satisfying the curvature condition ${\bm \sigma}_{2k}(A)=0$. This clearly suggests that a notion of {\em extrinsic} Lovelock black hole should be available in this setting and that a corresponding black hole uniqueness result in the line of Theorem \ref{main} should hold true under a suitable ellipticity condition. As in the case discussed in the bulk of this paper, in order to carry out this program we should be able to establish a version of Theorem \ref{al:rem:res} in this hyperbolic context. We also mention that a Penrose-type inequality for asymptotically hyperbolic graphs suitably embedded in $\mathbb H^{n+1}$ has been proved in \cite{ge2015gbc}; 
		the corresponding result in the Einsteinian setting ($k=1$) has been established in
		\cite{de2016alexandrov}, with a previous (non-sharp) contribution appearing in \cite{dahl2013penrose}. Thus, the appropriate version of Theorem \ref{rig:char} is also expected to hold true in this setting. We hope to address the issues raised in this section elsewhere.

		\bibliographystyle{alpha}
		\bibliography{ext-rig-sch-arxiv-v2}	
		
			%
			%

	\end{document}